\documentclass[11pt]{article}
\usepackage{amsmath}
\usepackage{amssymb}
\usepackage{amsthm}

\parskip=\medskipamount
\def\ad{\mathop{\mathrm{ad}}\nolimits}

\def\clift#1{#1^{\scriptscriptstyle{\mathrm{C}}}}
\def\hlift#1{#1^{\scriptscriptstyle{\mathrm{H}}}}
\def\vlift#1{#1^{\scriptscriptstyle{\mathrm{V}}}}
\def\hchecklift#1{#1^{\scriptscriptstyle{\check{\mathrm{H}}}}}

\def\C{\mathcal{C}}
\def\D{\mathcal{D}}
\def\fpd#1#2{{\displaystyle\frac{\partial #1}{\partial #2}}}
\def\H{\mathcal{H}}
\def\V{\mathcal{V}}
\def\S{\mathcal{S}}

\def\R{\mathbb{R}}

\def\vf#1{\frac{\partial}{\partial #1}}

\def\g{\mathfrak{g}}
\def\h{\mathfrak{h}}

\def\conn#1#2#3{\setbox1=\hbox{$\scriptstyle{#2}{#3}$}%
\setbox2=\hbox to\wd1{$\hfil\scriptstyle{#1}\hfil$}
\Gamma^{\!\box2}_{\!\box1}}
\def\barconn#1#2#3{\setbox1=\hbox{$\scriptstyle{#2}{#3}$}%
\setbox2=\hbox to\wd1{$\hfil\scriptstyle{#1}\hfil$}
\check{\Gamma}^{\!\box2}_{\!\box1}}

\def\onehalf{{\textstyle\frac12}}

\def\hook{{\mathchoice
{\vrule height 0pt depth 0.4pt width 3pt \vrule height 5pt depth 0.4pt
\kern 3pt}
{\vrule height 0pt depth 0.4pt width 3pt  \vrule height 5pt depth
0.4pt\kern 3pt}
{\vrule height 0pt depth 0.2pt width 1.5pt  \vrule height 3pt depth
0.2pt width 0.2pt \kern 1pt}
{\vrule height 0pt depth 0.2pt width 1.5pt  \vrule height 3pt depth
0.2pt width 0.2pt \kern 1pt} }}

\def\Langle{\langle\!\langle}
\def\Rangle{\rangle\!\rangle}

\newtheorem{prop}{Proposition}

\RequirePackage{color}
{\egroup}

\begin{document}

\title{Reduction of invariant constrained systems using anholonomic frames}

\author{M.\ Crampin and T.\ Mestdag\\
Department of Mathematics, Ghent University\\
Krijgslaan 281, S9, B--9000 Gent, Belgium}

\date{}

\maketitle

{\small {\bf Abstract.} We analyze two reduction methods for
nonholonomic systems that are invariant under the action of a Lie
group on the configuration space.  Our approach for obtaining the
reduced equations is entirely based on the observation that the
dynamics can be represented by a second-order differential equations
vector field and that in both cases the reduced dynamics can be
described by expressing that vector field in terms of an appropriately
chosen anholonomic frame.  \\[2mm]
}

\section{Introduction}

In a number of recently published papers
\cite{Routh,rraspects,nonholvak,invlag} we have developed a
distinctive geometric approach to the study of regular Lagrangian
dynamical systems, and especially to the problem of formulating
reduced equations for systems which are invariant under the action of
a symmetry group.  The main distinctive features of our approach are,
firstly, the formulation of the Euler-Lagrange equations in a way
which is well adapted to the idea that their function is to determine
a vector field on the velocity phase space which is of second-order
type (so that the differential equations which determine its integral
curves are of second order in the configuration space coordinates),
and yet is completely coordinate independent; and secondly, the
consistent use of anholonomic frames and their associated
quasi-velocities.  In this paper we shall extend these ideas to cover
Lagrangian systems subject to nonholonomic linear constraints, for
which both the Lagrangian function and the constraint distribution are
invariant. Such constraints arise naturally in the context of
systems with rigid bodies rolling without slipping over a surface or
possessing a contact point with the surface in the form of a knife
edge.  A classical reference for the dynamics of mechanical systems
with nonholonomic constraints is the book by Ne\u{\i}mark and Fufaev
\cite{Neimark}.  The recent books
\cite{Bloch,Cortes,Cushman} contain many references to
different modern geometric approaches to the theory. We will work with autonomous
systems; for formulations of the nonholonomic dynamics in a
time-dependent set-up see e.g.\ \cite{Olga,Willy}.

The formulation of the Euler-Lagrange equations mentioned above goes
as follows. We consider a Lagrangian system over a differentiable
manifold $Q$ (configuration space). The Lagrangian $L$ is a
function on the tangent bundle $\tau:TQ\to Q$ (velocity phase space); it is
regular if its Hessian with respect to the fibre coordinates is
nonsingular. The following proposition holds \cite{nonholvak}.

\begin{prop}\label{Prop1}
Let $L$ be a regular Lagrangian on $TQ$.  There is a unique
second-order differential equation  field $\Gamma$ such that
\[
\Gamma(\vlift{Z}(L))-\clift{Z}(L)=0
\]
for all vector fields $Z$ on $Q$. Moreover, $\Gamma$ may be determined
from the equations
\[
\Gamma(\vlift{X_i}(L))-\clift{X_i}(L)=0, \quad i=1,2,\ldots n=\dim Q,
\]
for any frame $\{X_i\}$ on $Q$ (which may be a coordinate frame or may be
anholonomic).
\end{prop}
\noindent Here $\vlift{Z}$ and $\clift{Z}$ are respectively the
vertical lift and the complete or tangent lift of $Z$ to $TQ$ (we
refer to \cite{CP} for the most common notions on tangent bundle
geometry).  The formulation in terms of the frame $\{X_i\}$ leads
directly to Hamel's equations
\[
\Gamma\left(\fpd{L}{v^i}\right)-X_i^j\fpd{L}{q^j}+R^j_{ik}v^k\fpd{L}{v^j}=0,
\]
where $X_i=X_i^j\partial/\partial q^j$; the $v^i$ are the
quasi-velocities associated with the frame; and the coefficients
$R^k_{ij}$ are defined by $[X_i,X_j]=R_{ij}^kX_k$ and are collectively
called the object of anholonomity of the frame.  We refer the reader
to Section~\ref{form} for more details.

The equations determining the dynamics of a regular system subject to
nonholonomic linear constraints admit a rather similar formulation.
The constraints may be specified in either of two equivalent ways:\ as
a distribution $\D$ on $Q$ (the constraint distribution), or as a
submanifold $\C$ of $TQ$ (the constraint submanifold).  The two are
related as follows: $\C=\{(q,u)\in TQ:u\in \D_q\subset T_qQ\}$.  We
assume that the dimension of each $\D_q$, and equivalently the fibre
dimension of $\C_q$, is constant and equal to $m$.  A vector field
$\Gamma$ on $\C$ is said to be of second-order type if it satisfies
$\tau_{*(q,u)}\Gamma=u$ for all $(q,u)\in\C$.  A Lagrangian function
$L$ is said to be regular with respect to $\D$ if for any
local basis $\{X_\alpha\}$ of $\D$, $1\leq\alpha\leq m$, the symmetric
$m\times m$ matrix whose entries are
$\vlift{X_\alpha}(\vlift{X_\beta}(L))$ (functions on $\C$) is
nonsingular. In \cite{nonholvak} we proved the following proposition.

\begin{prop}
Let $L$ be a Lagrangian on $TQ$ which is regular with
respect to $\D$.  Then there is a unique vector field $\Gamma$ on $\C$
which is of second-order type, is tangent to $\C$, and is such that on
$\C$
\[
\Gamma(\vlift{Z}(L))-\clift{Z}(L)=0
\]
for all $Z\in\D$.  Moreover, $\Gamma$ may be determined from the
equations
\[
\Gamma(\vlift{X_\alpha}(L))-\clift{X_\alpha}(L)=0, \quad
\alpha=1,2,\ldots m,
\]
on $\C$, where $\{X_\alpha\}$ is any local basis for $\D$.
\end{prop}
This is our version of the Lagrange-d'Alembert principle (see
\cite{Bloch,Cortes} for other versions); the vector field $\Gamma$ is
the dynamical field of the constrained system.

The formal similarity between the standard Euler-Lagrange equations and the
Lagrange-d'Alembert equations in these formulations is self-evident.
We shall exploit this similarity in deriving the reduced equations
for an invariant constrained system:\ as we shall show, to obtain
those equations it is enough to follow the reduction procedure for an
invariant unconstrained system, while restricting attention to the
constraint submanifold.

There are in fact two well-known ways of reducing the equations of an
invariant unconstrained Lagrangian system.  One method does not even
take the Lagrangian structure of the system into account and simply
involves factoring out by the action of the group, and leads to the
so-called Lagrange-Poincar\'{e} equations; this is described in e.g.\
\cite{Cendra,DMM, invlag}.  The second does take advantage of momentum
conservation; its first step is to restrict to a level set of
momentum, and this is followed by a reduction with respect to the
invariance group of the chosen value of the momentum.  This is a
generalized version of Routh's procedure; it is discussed in
\cite{Routh,MRS}.  For a constrained system, however, these two
methods are not equally applicable.  This is because, although the
constraint distribution is invariant under the symmetry group, it is
not usually the case that any fundamental vector field of the action
belongs to it.  There is consequently no conservation of momentum, and
no possiblility of Routh-type reduction.  The greater part of this
paper is therefore devoted to the adaptation of Lagrange-Poincar\'{e}
reduction to constrained systems.  This we discuss in full
generality:\ whereas many other papers (\cite{BKMM,AMZ,CMR2,MeLa} for
example) restrict their attention to the case in which at each point
$q$ of $Q$ the constraint distribution $\D_q$ and the tangent space
$\V_q$ to the orbit of the action together span $T_qQ$, we make no
such so-called `dimension assumption'; our only requirement is that
$\D_q\cap\V_q$ has constant dimension.

Though Routh-type reduction is not possible in general, it can arise
in particular cases, where there is a Lie subgroup $H$ of the symmetry
group $G$, necessarily normal, with Lie algebra $\h$, such that for
all $\xi\in\h$ the corresponding fundamental vector field
$\tilde{\xi}$ lies in $\D$.  Symmetries belonging to $H$ are said to
be horizontal. We devote a separate section to the discussion of this
case.

The paper is laid out as follows.  In the following section we deal
with the fundamental definitions and results concerning invariance of
a constrained system under the free and proper action of a Lie group
$G$ on $Q$, leading to a version of the Atiyah sequence for such a
system.  In Section~3 we give a resum\'{e} in general terms of the
Lagrange-Poincar\'{e} reduction procedure, and show
how it may be adapted to the case of an invariant constrained
system.  In Section~4 we derive explicit Hamel-type
formulae, in terms of a (possibly) anholonomic frame, for the
Lagrange-Poincar\'{e} equations and the
Lagrange-d'Alembert-Poincar\'{e} equations successively.  In Section~5
we discuss Routh-type reduction for systems with horizontal
symmetries.

\section{Invariance of nonholonomic systems}\label{inv}

Assume from now on that a connected Lie group $G$ acts in a free and
proper way on the left on the configuration manifold $Q$.  Then $\pi:Q
\to Q/G$ is a principal fibre bundle.  The action $\psi_g$ on $Q$
induces an action $T\psi_g$ on $TQ$.  We will write $\tilde A$ for
the infinitesimal generator of the action on $Q$, associated to a Lie
algebra element $A \in\g$.  Then $\clift{\tilde A}$ is an
infinitesimal generator for the action on $TQ$.  As in e.g.\
\cite{AMZ}, we say that the nonholonomic system is invariant under
$G$, or that it admits $G$ as a symmetry group, if both the Lagrangian
$L$ and the constraint submanifold $\C$ of the system are invariant
under the induced action of $G$ on $TQ$.

\begin{prop}
The constraint submanifold $\C\subset TQ$ is invariant under $T\psi$ if and
only if the constraint distribution $\D$ on $Q$ is invariant under $\psi$.
\end{prop}

\begin{proof}
Since $\C=\{(q,u):u\in \D_q\}$, $\C$ is invariant under
$T\psi$ if and only if for every $q\in Q$, $u\in\D_q$ and $g\in G$,
$\psi_{g*}u\in\D_{\psi_g(q)}$; that is to say, for every $q\in Q$
and $g\in G$, $\D_{\psi_g(q)}=\psi_{g*}\D_q$.
\end{proof}

\begin{prop} If $L$ is regular with respect to $\D$ the vector field
$\Gamma$ is invariant under the induced action of $G$ on
$\C$.\end{prop}

\begin{proof} For any $A \in\g$ and $Z\in\D$
we have
\begin{align*}
0&= \clift{\tilde A} \big( \Gamma(\vlift{Z}(L))
-\clift{Z}(L) \big)\\
&= [\clift{\tilde A},\Gamma](\vlift{Z}(L))
- \Gamma (\clift{\tilde A} (\vlift{Z}(L)))
- [\clift{\tilde A}, \clift{Z}](L)\\
& = [\clift{\tilde A},\Gamma](\vlift{Z}(L))
-\Gamma ([\tilde A,Z\vlift](L)) - [\tilde A,Z\clift](L).
\end{align*}
Now $[\tilde{A},Z]\in\D$, due to the assumed invariance of $\D$.  By
the Lagrange-d'Alembert equation the last two terms above vanish.  On
the other hand, the bracket $[\clift{\tilde A},\Gamma]$ is vertical.
This is certainly true for a second-order differential equation field
on $TQ$, by a simple calculation in coordinates.  Now $\Gamma$ is a
second-order differential equation field on $\C$; but we can evidently
extend it to a second-order differential equation field on a
neighbourhood of $\C$ in $TQ$.  Since both $\clift{\tilde A}$ and
$\Gamma$ are tangent to $\C$, so also is their bracket.  So on $\C$,
$[\clift{\tilde A},\Gamma]$ is independent of the choice of
extension, and is vertical.  It follows from the fact that the
equation $[\clift{\tilde A},\Gamma](\vlift{Z}(L))=0$ holds for all
$Z\in\D$, and the assumption that $L$ is regular with respect to $\D$,
that $[\clift{\tilde A},\Gamma]=0$.  This may easily be seen by
expressing everything in terms of the vertical lifts of a local basis
for $\D$.  But $[\clift{\tilde A},\Gamma]=0$ is the infinitesimal
condition for $\Gamma$ to be invariant.\end{proof}

Since $\Gamma$ is invariant, it reduces to a vector field
$\check\Gamma$ on $\C/G$.  Our main overall aim in this paper is to
show how to determine this reduced vector field. We begin however by
considering some aspects of the structure of nonholonomic systems
which are invariant in the sense defined above.

As we have already noted, $\pi:Q \to Q/G$ is a principal fibre bundle.
Since $\D$ is invariant it defines a distribution $\bar{\D}$ on $Q/G$
by $\bar{\D}_{\pi(q)}=\pi_*(\D_q)$; this is well-defined because
$\pi\circ\psi_g=\pi$.  Let us assume that $\bar{\D}$ has constant
dimension.  Then $Q/G$ is equipped with a regular distribution
$\bar{\D}$.  Denote the corresponding submanifold (indeed vector
subbundle) of $T(Q/G)$ by $\bar{\C}$.

Let $\V_q = \ker\pi_{*q}$.  Note that
$\ker\pi_{*q}|_{\D_q}=\D_q\cap\V_q$.  Let us denote it by $\S_q$.
Evidently $\S$ is an invariant distribution on $Q$, which is of
constant dimension by the corresponding assumption for $\bar{\D}$.
Since $\S_q\subset \V_q$ for each $q\in Q$, we may identify $\S_q$
with a vector subspace $\g^q$ of $\g$, where $\g^q=\{A \in\g \,|\,
{\tilde A}_q \in\S_q\}$.  In terms of $TQ$, we can express $\g^q$ as
follows.  For $w\in T_qQ$, $w\in\D_q$ if and only if $\vlift{w}$ is
tangent to $\C$; thus $A\in\g^q$ if and only if
$\vlift{\tilde{A}_q}$ is tangent to $\C$.  Since (see e.g.\
\cite{CP,MRbook})
\[
\psi_{g*}\left(\tilde{A}_q\right)=
\left(\widetilde{\ad(g^{-1})A}\right)_{\psi_g(q)}
\]
we have
\[
\g^{\psi_g(q)}=\ad(g^{-1})\g^q.
\]
Consider $\g^\D=\{(q,A)\,|\,A\in\g^q\}$. There is an action of
$G$ on $\g^\D$ given by
\[
(q,A)\mapsto(\psi_g(q),\ad(g^{-1})A).
\]
On taking the quotient we obtain a vector bundle over $Q/G$, say
$\bar{g}^\D$, which is a vector subbundle of $\bar{\g} = (Q\times\g)/G
\to Q/G$, the adjoint bundle associated with the principal $G$-bundle
$Q$.

\begin{prop}\label{Atiyahprop}
We have the following short exact sequence of vector bundles over $Q/G$:
\[
0 \to \bar{\g}^\D \to \C/G \to \bar{\C} \to 0.
\]
\end{prop}

This is a version for constrained systems of the
so-called Atiyah sequence (see e.g.\ \cite{rraspects,DMM}),
\[
0 \to \bar{\g} \to TQ/G \to T(Q/G) \to 0.
\]
Each term in the sequence of the proposition is a subbundle of the
corresponding term in the Atiyah sequence.  In the next section we
will use this observation when we divide the reduced equations for an
invariant nonholonomic system into two sets.

\section{Lagrange-Poincar\'{e} reduction:\ generalities}\label{gen}

\subsection{Standard Lagrange-Poincar\'e reduction}

Before considering reduction of invariant nonholonomic systems we
discuss Lagrange-Poincar\'{e} reduction of the standard Euler-Lagrange
equations.  Recall the Euler-Lagrange equations as they appear in
Proposition~\ref{Prop1}.  Assume that $L$ is $G$-invariant:\ then so
is $\Gamma$; it reduces to a vector field $\check\Gamma$ on $TQ/G$,
which we want equations for --- so-called reduced equations.

There is a sense in which the reduction of the Euler-Lagrange
equations is immediate.  The key step is to rewrite them in
$G$-invariant form.  It is enough to take $Z$ to be invariant.  Then
$\vlift{Z}$ and $\clift{Z}$ are invariant under the induced action of
$G$ on $TQ$, and so define vector fields on $TQ/G$, which we denote by
$\vlift{\check{Z}}$ and $\clift{\check{Z}}$, though of course they are
not vertical or complete lifts.  The function $\clift{Z}(L)$ is
invariant, and so defines a function on $TQ/G$, which is just
$\clift{\check{Z}}(l)$ (where $l$ is the reduced function of $L$ on $TQ/G$);
likewise for $\vlift{Z}(L)$.  Then the reduced equations are simply
\[
\check\Gamma(\vlift{\check{Z}}(l))-\clift{\check{Z}}(l)=0,
\]
on $TQ/G$, for all invariant vector fields $Z$ on $Q$; they are called the
Lagrange-Poincar\'{e} equations.

However, we can be more explicit. The Euler-Lagrange equations can be
divided into two sets, according to whether we take $Z$ to be tangent
to the fibres of $\pi$ or transverse to them.

{\bf The Lagrange-Poincar\'e equation for momentum.} Consider first
the Euler-Lagrange equation $\Gamma(\vlift{Z}(L))-\clift{Z}(L)=0$ for
any vector field $Z$ which is vertical with respect to $\pi:Q\to Q/G$.
Such a vector field is determined by a $\g$-valued function $\zeta$ on
$Q$, where $Z_q=\widetilde{\zeta(q)}_q$.  The momentum $p$ is a
$\g^*$-valued function on $TQ$, which is $G$-equivariant under the
usual action of $G$ on $TQ$ and the coadjoint action on $\g^*$.

Take first $A\in\g$.  We have
\[
\vlift{\tilde{A}}(L)=\langle A,p\rangle
\]
(as real-valued functions on $TQ$; the angle brackets denote the
pairing of $\g$ and $\g^*$).  The conservation of momentum is
just $\Gamma\langle A,p\rangle=0$ (the Euler-Lagrange equation
with $Z=\tilde{A}$):\ or, since $A$ is constant and arbitrary,
$\Gamma(p)=0$.

Now consider $\langle\zeta,p\rangle$:\ applying Leibniz' rule we have
\[
\Gamma\langle\zeta,p\rangle=
\langle\Gamma(\zeta),p\rangle+\langle\zeta,\Gamma(p)\rangle
=\langle\dot{\zeta},p\rangle
\]
(using the fact that $\Gamma(f)=\dot{f}$ for any function $f$ on $Q$).
We claim that the (almost tautological) equation
$\Gamma\langle\zeta,p\rangle=\langle\dot{\zeta},p\rangle$ is the
Euler-Lagrange equation for vertical $Z$.  We compute $\vlift{Z}$ and
$\clift{Z}$ in terms of $\zeta$, as follows.  Take a basis $\{E_r\}$
of $\g$, and set $\zeta=\zeta^rE_r$, where the coefficients $\zeta^r$
are functions on $Q$.  Then $Z=\zeta^r\tilde{E}_r$, and so
\[
\vlift{Z}=\zeta^r\vlift{\tilde{E}_r},\quad
\clift{Z}=\zeta^r\clift{\tilde{E}_r}+\dot{\zeta}^r\vlift{\tilde{E}_r}.
\]
Then
\[
\vlift{Z}(L)=\zeta^rp_r=\langle\zeta,p\rangle,\quad
\clift{Z}(L)=\dot{\zeta}^rp_r=\langle\dot{\zeta},p\rangle
\]
as claimed.

For reduction we need to take $Z$ to be $G$-invariant:\ it will be so
if and only if $\zeta$ is $G$-equivariant (with now the adjoint action
on $\g$); this amounts to taking a section of the adjoint bundle $\bar
\g \to Q/G$ over $Q$.  It is then clear that $\langle\zeta,p\rangle$
will be invariant.  It must then be the case (from the Euler-Lagrange
equation and the invariance of $\Gamma$) that
$\langle\dot{\zeta},p\rangle$ is invariant.  This can be shown
directly too, in various ways.  Here is a vector field version.  The
$G$-equivariance of $p$ can be expressed as
\[
\langle B,\clift{\tilde{A}}(p)\rangle=-\langle [A,B],p\rangle,
\]
where $A,B\in\g$ and $[A,B]$ is their bracket in $\g$.  The assumed
$G$-equivariance of $\zeta$ is just $\tilde{A}(\zeta)=[A,\zeta]$
(again, bracket in $\g$).  Then
\begin{align*}
\clift{\tilde{A}}\langle\dot{\zeta},p\rangle
&=\langle\clift{\tilde{A}}(\dot{\zeta}),p\rangle
+\langle\dot{\zeta},\clift{\tilde{A}}(p)\rangle\\
&=\left\langle\frac{d}{dt}(\tilde{A}(\zeta)),p\right\rangle
-\langle[A,\dot{\zeta}],p\rangle\\
&=\left\langle\frac{d}{dt}([A,\zeta]),p\right\rangle
-\langle[A,\dot{\zeta}],p\rangle=0,
\end{align*}
using the obvious fact that $\clift{X}(\dot{f})=d/dt(X(f))$.

Since $\langle\zeta,p\rangle$ and $\langle\dot{\zeta},p\rangle$ are
invariant they define functions on $TQ/G$, which we denote by
$\Langle\zeta,p\Rangle$ and $\Langle\dot{\zeta},p\Rangle$. The
corresponding reduced equation is
\[
\check\Gamma\Langle\zeta,p\Rangle=\Langle\dot{\zeta},p\Rangle.
\]
We call it the Lagrange-Poincar\'{e} equation for momentum.

{\bf The horizontal Lagrange-Poincar\'e equation.} To obtain an
invariant Euler-Lagrange equation corresponding to the transverse
directions we can use a principal connection on $\pi:Q\to
Q/G$, or, equivalently, a splitting of the Atiyah sequence.
Suppose we have such a connection:\ for any vector field $Y$ on $Q/G$
let $\hlift{Y}$ be its horizontal lift to $Q$; then the transverse
(let's call it horizontal) Euler-Lagrange equation is
\[
\Gamma(\vlift{(\hlift{Y})}(L))-\clift{(\hlift{Y})}(L)=0.
\]
Incidentally, this expression is $C^\infty(Q/G)$-linear in $Y$.
Moreover, each term is $G$-invariant.

We will express this equation in a different way.  First we recall the
construction of the Vilms connection from \cite{rraspects} (which is a
special case of a more general construction in \cite{Vilms}).  The
complete lift of a type $(1,1)$ tensor $T$ on $Q$ is given by
\cite{CP}
\[
\clift{T}(\clift{X})=\clift{T(X)},\quad
\clift{T}(\vlift{X})=\vlift{T(X)}.
\]
The original connection on $\pi:Q\to Q/G$ can be represented by a type
$(1,1)$ tensor field $\omega$, so that $\omega(X)=0$ if and only if
$X$ is horizontal, $\omega(V)=V$ for $V$ vertical.  Then the type
$(1,1)$ tensor on $TQ$ defining the Vilms connection is just
$\clift{\omega}$; moreover, it is invariant under the action of $G$ on
$TQ$.  From the defining relations of the complete lift of a type
$(1,1)$ tensor field above, one easily concludes that the horizontal
distribution defined by the Vilms connection is spanned by the
complete and vertical lifts of the horizontal vector fields of the
original connection.

Next, a remark about complete and vertical lifts.  Let $\phi:M\to N$
be a smooth map, and suppose that vector fields $U$ on $M$ and $V$ on
$N$ are $\phi$-related.  Then $\clift{U}$ and $\clift{V}$ are
$T\phi$-related, and likewise $\vlift{U}$ and $\vlift{V}$.  One can
easily prove this by considering the flows of the involved vector
fields. 

Consider now $\vlift{(\hlift{X})}$, for any vector field $X$ on $Q/G$.
It is horizontal with respect to the Vilms connection.  By the
previous remark, since $\hlift{X}$ is $\pi$-related to $X$,
$\vlift{(\hlift{X})}$ is $T\pi$-related to $\vlift{X}$.  Thus
$\vlift{(\hlift{X})}$ is the horizontal lift with respect to the Vilms
connection of the vector field $\vlift{X}$ on $T(Q/G)$.  With some
abuse of notation we may write
$\vlift{(\hlift{X})}=\hlift{(\vlift{X})}$ (warning:\ V and H have
different meanings either side of the equality sign).  Likewise for
$\clift{X}$.  So we can rewrite the horizontal Euler-Lagrange equation
above as follows:
\[
\Gamma(\hlift{(\vlift{Y})}(L))-\hlift{(\clift{Y})}(L)=0
\]
for all $Y$ on $Q/G$.

The Vilms connection is $G$-invariant, so each of
$\hlift{(\vlift{Y})}$ and $\hlift{(\clift{Y})}$ is an invariant vector
field on $TQ$, and so each passes to the quotient to define vector
fields $\hchecklift{(\vlift{Y})}$ and $\hchecklift{(\clift{Y})}$ on
$TQ/G$. The reduced horizontal Euler-Lagrange equation is
\[
\check\Gamma(\hchecklift{(\vlift{Y})}(l))-\hchecklift{(\clift{Y})}(l)=0
\]
on $TQ/G$, for all $Y$ on $Q/G$. We call it the horizontal
Lagrange-Poincar\'{e} equation.

In the next step, we can write this equation as
\[
\check\Gamma(\langle \vlift{Y},\hchecklift{d}l\rangle)-
\langle \clift{Y},\hchecklift{d}l\rangle=0.
\]
Here $\hchecklift{d}l$ is a 1-form along the projection $TQ/G\to
T(Q/G)$ (or a 1-form on $TQ/G$ which is semi-basic with respect to
that projection), such that for any vector field $W$ on $T(Q/G)$,
$\langle W,\hchecklift{d}l\rangle=\hchecklift{W}(l)$.  Since
$\vlift{Y}$ and $\clift{Y}$ are actually the lifts from $Q/G$ to
$T(Q/G)$ this looks very much like an Euler-Lagrange equation on
$T(Q/G)$ (in fact, it would be one if $\hchecklift{d}l$ was replaced
by the exterior derivative $d$ on $T(Q/G)$).

We have divided the reduced equations in two sets, in accordance with
the decomposition of $TQ/G$ given by the Atiyah sequence.  We conclude
therefore: 
\begin{prop}
The Lagrange-Poincar\'{e} equations are given by
\begin{align*}
&\check\Gamma\Langle\zeta,p\Rangle=\Langle\dot{\zeta},p\Rangle\\
&\check\Gamma(\langle \vlift{Y},\hchecklift{d}l\rangle)-
\langle \clift{Y},\hchecklift{d}l\rangle=0,
\end{align*}
where $\zeta$ is any $G$-equivariant $\g$-valued function on $TQ$ and
$Y$ is any vector field on $Q/G$.
\end{prop}

\subsection{Lagrange-Poincar\'{e}-type reduction of nonholonomic
systems}\label{nonholgen}

Again, there is a sense in which the reduction of the
Lagrange-d'Alembert equations is immediate.  These equations say that
on $\C$
\[
\Gamma(\vlift{Z}(L))-\clift{Z}(L)=0
\]
for all $Z\in\D$.  We assume that $\D$ is $G$-invariant.  It is again
enough to take $Z$ to be invariant.  The function $\clift{Z}(L)$ is
invariant, and so defines a function on $TQ/G$, which is just
$\clift{\check{Z}}(l)$ (where $l$ is the reduced function of $L$).
Likewise for $\vlift{Z}(L)$; however, since $Z\in\D$, $\vlift{Z}$ is
tangent to $\C$ and so we can replace $l$ by $l_c$ in the first term.
Then the reduced equation is simply
\[
\check\Gamma(\vlift{\check{Z}}(l_c))-\clift{\check{Z}}(l)=0
\]
on $\C/G$, for all $Z\in\D$; of course the second term must be
understood as the restriction of that function on $TQ/G$ to the
submanifold $\C/G$.

The reduced equations are now called the
Lagrange-d'Alembert-Poincar\'{e} equations.  Like the
Lagrange-Poincar\'{e} equations they can be split into two sets,
corresponding to the vertical and horizontal parts of $\D$.  Much as
before, these two sets are dictated by the version of the Atiyah
sequence we found in Proposition~\ref{Atiyahprop}.

{\bf The Lagrange-d'Alembert-Poincar\'{e} equation for momentum.}
First, we consider vertical vector fields in $\D$, that is, vector
fields $Z$ in $\S$.  Every such vector field defines a $\g$-valued
function $\zeta$ on $Q$, where now $\zeta(q)\in\g^q$.  Since $L$ is
$G$-invariant we may define the momentum map $p:TQ\to\g^*$, as usual,
and it is $G$-equivariant with respect to the coadjoint action on
$\g^*$.  However, we now have no reason to suppose that in general
$p$, or any component of it, is conserved.  On the other hand, the
argument that leads to the formulae
$\vlift{Z}(L)=\langle\zeta,p\rangle$ and
$\clift{Z}(L)=\langle\dot{\zeta},p\rangle$ still holds good, and we
conclude that the Lagrange-d'Alembert equation for $Z\in\S$ can be
written
\[
\Gamma\langle\zeta,p\rangle=\langle\dot{\zeta},p\rangle
\]
on $\C$.  We conclude further that the following weakened version of
conservation of momentum for a constrained system holds:\
$\langle\zeta,\Gamma(p)\rangle=0$ for all $\g$-valued functions
$\zeta$ such that $\zeta(q)\in\g^q$; that is to say, for all
$(q,u)\in\C$, $\Gamma_{(q,u)}(p)\in(\g^q)^\perp$.  Of course, if it
should happen that for some $\zeta$, $\langle\dot{\zeta},p\rangle=0$
on $\C$ then $\langle \zeta,p\rangle$ will be conserved.  In
particular, this will occur if $\S$ contains a fundamental vector
field of the action, that is, if there is some $A\in\g$ such that
$A\in\g^q$ for all $q\in Q$:\ then $\langle A,p\rangle$ (the
$A$-component of momentum) will be conserved.

To obtain a $G$-invariant vector field $Z\in\S$ we must take $\zeta$
to be $G$-equivariant under the adjoint action on $\g$. Then we have
the reduced equation
\[
\check\Gamma\Langle\zeta,p\Rangle=\Langle\dot{\zeta},p\Rangle
\]
on $\C/G$, where $\zeta(q)\in\g^q$.

{\bf The horizontal Lagrange-d'Alembert-Poincar\'{e} equation.} To
obtain the reduced equation corresponding to the horizontal part of
$\D$ we need a splitting of the modified Atiyah sequence of
Proposition~\ref{Atiyahprop}.  One may derive such a splitting from a
principal connection on $\pi:Q\to Q/G$ with the property that the
horizontal lift of $\bar{\D}$ is contained in $\D$.  Such a connection
can be constructed locally, by defining its horizontal subspaces as
follows.  Take a local section of $\pi$.  For every $q$ in the image
of the section choose some complement to $\S_q$ in $\D_q$ and extend
it to a complement of $\V_q$ in $T_qQ$, smoothly over the section.
Finally, extend the result along the fibres by the action of $G$.

The reduced equation is
\[
\check\Gamma(\langle \vlift{Y},\hchecklift{d}l\rangle)-
\langle \clift{Y},\hchecklift{d}l\rangle=0
\]
as before, but now with $Y\in\bar{\D}$.

The conclusion of this section is therefore:
\begin{prop}
The Lagrange-d'Alembert-Poincar\'e equations are given
by
\begin{align*}
&\check\Gamma\Langle\zeta,p\Rangle=\Langle\dot{\zeta},p\Rangle\\
&\check\Gamma(\langle \vlift{Y},\hchecklift{d}l\rangle)-
\langle \clift{Y},\hchecklift{d}l\rangle=0
\end{align*}
on $\C/G$,
where $\zeta(q)\in\g^q$ and $Y\in\bar{\D}$.
\end{prop}

\section{Lagrange-Poincar\'{e}-type reduction:\ formulae}\label{form}

The versions of the reduced equations given in the previous section
are elegant and instructive.  In the literature one may find other
geometric approaches to obtain the reduced equations (see e.g.\
\cite{Cendra,DMM} for the Lagrange-Poincar\'e equations and
\cite{Bloch,AMZ,CMR2,MeLa} for the Lagrange-d'Alembert-Poincar\'e
equations).  In the interest of comparison, we shall now formulate
local versions of our reduced equations.  Our method is entirely based
on the use of suitably adapted anholonomic frames on $Q$.
Unsurprisingly, the version of the Lagrange-d'Alembert-Poincar\'{e}
equations we finally obtain will combine elements of both the
Lagrange-d'Alembert and the Lagrange-Poincar\'{e} equations, so we
first deal with each of those cases separately.

But even before doing so, it is worth recalling the basic formulae
relating to anholonomic frames.  Let $\{X_i\}$ be an anholonomic frame
on $Q$ and $v^i$ the quasi-velocities corresponding to that frame.
(The quasi-velocities $v^i$ are not to be confused with the canonical
fibre coordinates associated with the $x^i$; the coordinates
$(x^i,v^i)$ are to that extent unnatural.) A second-order field
$\Gamma$ can then be written in the form
\[
\Gamma = v^i \clift{X_i} + f^i\vlift{X_i}.
\]
We write $[X_i,X_j]=R_{ij}^kX_k$, where the functions
$R^k_{ij}$ are collectively called the object of anholonomity. Let
$v^i$ be the quasi-velocities corresponding to the frame:\ then
\[
\clift{X_i}(v^j)=-R_{ik}^jv^k, \qquad \vlift{X_i}(v^j)=\delta_i^j.
\]
In terms of coordinates $x^i$ on $Q$ we may write
\[
\clift{X_i}=X_i^j\vf{x^j}-R_{ik}^jv^k\vf{v^j},\qquad
\vlift{X_i}=\vf{v^i},
\]
where $X_i=X_i^j\partial/\partial x^j$.  The first term in the
expression for $\clift{X_i}$ is formally the same as $X_i$ itself, but
is of course a local vector field on $TQ$ rather than a vector field
on $Q$; we shall continue to denote it by $X_i$, though this is
strictly speaking an abuse of notation.  With this understood, the
Euler-Lagrange equations may be written in Hamel form,
\[
\Gamma\left(\fpd{L}{v^i}\right)-X_i(L)+R^j_{ik}v^k\fpd{L}{v^j}=0
\]
(as in e.g.\ \cite{AMZ}), which is the form we had announced in our
Introduction.

\subsection{The Lagrange-Poincar\'{e} equations}

We take the frame $\{X_i\}$ on $Q$ to be invariant under the action of
$G$, and to be of the form $\{X_r,X_I\}$ where the $X_r$ are vertical
and such that their values at any point $q$ form a basis for the vertical
vectors at $q$.  Then the $X_I$ are invariant and transverse to the
fibres of $Q\to Q/G$, and may be considered as the horizontal lifts of
their projections $Y_I$ to $Q/G$, with respect to some principal
connection $\omega$; note that $\{Y_I\}$ is a frame for $Q/G$, in
general anholonomic.

Let $\{E_r\}$ be a basis of $\g$ and let $\tilde{E}_r$ be the
fundamental vector fields of the action corresponding to this basis.
We have $[\tilde{E}_r,\tilde{E}_s]=-C_{rs}^t\tilde{E}_t$ where the
coefficients $C_{rs}^t$ are the structure constants of $\g$ with
respect to the given basis.  A vector field on $Q$ is invariant if and
only if all $[X,{\tilde E}_r] = 0$.

The vector fields $X_r$ could be obtained by taking a local section of
$Q\to Q/G$, choosing a basis of vertical vectors at each point of the
section varying smoothly over it, and using the $G$-action to define
the vector fields along the fibres.  If one chooses the initial values
of the $X_r$ to be ${\tilde E}_r$ we obtain the vector fields ${\hat
E}_r$ that we have used in previous publications
\cite{rraspects,invlag}.  However, in view of what we need for the
next section on constrained systems, it will be convenient to work in
greater generality already in this part of the section.

We may write $X_r=X_r^s\tilde{E}_s$ where the coefficient matrix is
nonsingular.  We have
\[
[\tilde{E}_r,X_s]=[\tilde{E}_r,X_s^t\tilde{E}_t]=
(\tilde{E}_r(X_s^t)-C_{ru}^tX_s^u)\tilde{E}_t,
\]
and so the necessary and sufficient condition for the $X_r$ to be
invariant is that the coefficients $X_s^t$ satisfy
\[
\tilde{E}_r(X_s^t)=C_{ru}^tX_s^u.
\]
It follows immediately that $X_r(X_s^t)=C_{uv}^tX_r^uX_s^v$, whence
\[
[X_r,X_s]=[X_r,X_s^t\tilde{E}_t]=C_{uv}^tX_r^uX_s^v\tilde{E}_t
=(\bar{X}^t_wC_{uv}^wX_r^uX_s^v)X_t
\]
where the overbar indicates the matrix inverse.  Since $X_r$, $X_s$
and $X_t$ are all invariant, so must the coefficient be.  We set
$\bar{C}^t_{rs}=\bar{X}^t_wC_{uv}^wX_r^uX_s^v$.  Then each
$\bar{C}^t_{rs}$ may be treated as a function on $Q/G$, and the
collection of such functions may be regarded as the structure
constants of $\g$, though expressed in terms of the $X_r$.  (If
$X_r=\hat{E}_r$ then $\bar{C}^t_{rs}=C_{rs}^t$.)

The equation $\tilde{E}_r(X_s^t)=C_{ru}^tX_s^u$ expresses the fact
that the $\g$-valued function $X_s^tE_t=\xi_s$ on $Q$ is
$G$-equivariant with respect to the adjoint action, and therefore
corresponds to a section of the adjoint bundle $\bar{\g}\to Q/G$; thus
the $\xi_r$ together form a local basis of sections of the adjoint
bundle.  Now
\[
X_I(\xi_r)=X_I(X_r^s)E_s=\bar{X}_t^sX_I(X_r^t)\xi_s;
\]
for convenience we shall write $\Upsilon_{Ir}^s$ for
$\bar{X}_t^sX_I(X_r^t)$.  We know that the functions $\Upsilon_{Ir}^s$
are $G$-invariant, and may therefore be regarded as functions on
$Q/G$.  Now $X_I$ is the horizontal lift of the vector
field $Y_I$ on $Q/G$ to $Q$, so we have
\[
\Upsilon_{Ir}^s\xi_s=\hlift{Y_I}(\xi_r).
\]
This means that the $\Upsilon_{Ir}^s$ are the connection coefficients,
with respect to the local basis $\{Y_I\}$ of vector fields on $Q/G$
and the local basis $\{\xi_r\}$ of sections of $\bar{\g}\to Q/G$, of
the connection induced by $\omega$ on the adjoint bundle (see e.g.\
\cite{rraspects} for more details).

Since the elements of the frame $\{X_i\}$ are invariant, so are their
brackets, and so is each $R_{ij}^k$; it may therefore be regarded as a
function on $Q/G$.  The following facts about the $R_{ij}^k$ are
important.
\begin{itemize}
\item The $R_{IJ}^K$ constitute the object of anholonomity of the
frame $\{Y_I\}$.  \item The vertical component of $[X_I,X_J]$, which
is $R_{IJ}^rX_r$, is closely related to the curvature of the
connection $\omega$:\ in fact the curvature, as a $\g$-valued
function, is $-R_{IJ}^sX^r_sE_r$.  We write $-K_{IJ}^r$ for $R_{IJ}^r$
as a reminder of this fact.
\item Since $[X_i,X_r]$ is always vertical, $R^I_{ir}=0$.  \item
$R_{Ir}^s=\Upsilon_{Ir}^s$.  \item $R_{rs}^t=\bar{C}_{rs}^t$.
\end{itemize}

We let $(v^r,v^I)$ be the the quasi-velocities corresponding to the
frame $\{X_r,X_I\}$; thus for $v\in T_qQ$, $v=v^rX_r(q)+v^IX_I(q)$.
But then $\pi_*v\in T_{\pi(q)}(Q/G)$ is given by
\[
\pi_*v=v^I\pi_*X_I(q)=v^IY_I(\pi(q)).
\]
From the invariance of the frame $\{X_i\}$ we conclude that the $v^i$
are invariant, and therefore constitute fibre coordinates on $TQ/G\to
Q/G$.  Since $\pi_*:T_qQ\to T_{\pi(q)}(Q/G)$ is surjective, we can
identify $v^I$ with $\pi^*w^I$, where the $w^I$ are the
quasi-velocities of the frame $\{Y_I\}$.

We now consider the Euler-Lagrange equation for momentum,
$\Gamma(\vlift{X_r}(L))-\clift{X_r}(L)=0$.  The components of momentum
are of course given by $p_r=\vlift{\tilde{E}_r}(L)$.  But we are
working in terms of the invariant basis $\{X_r\}$.  Let us set
$P_r=\vlift{X_r}(L)=X_r^sp_s=\langle \xi_r,p\rangle$; then
$P_r$ is invariant.  The Euler-Lagrange equation becomes
\[
\Gamma(P_r)=\dot{X}_r^sp_s=\bar{X}^s_t\dot{X}_r^tP_s,
\]
which is the component form of the momentum equation given in
Section~\ref{gen}, and is usually referred to as the momentum equation
in a moving basis \cite{Bloch,BKMM}.  Now in terms of quasi-velocities
\[
\frac{d}{dt}=v^iX_i=v^IX_I+v^rX_r,
\]
whence
\[
\bar{X}^s_t\dot{X}_r^t=\bar{X}^s_t(v^IX_I(X^t_r)+v^uX_u(X^t_r))
=\Upsilon_{Ir}^sv^I-\bar{C}_{rt}^sv^t,
\]
and we have
\[
\Gamma(P_r)=(\Upsilon_{Ir}^sv^I-\bar{C}_{rt}^sv^t)P_s.
\]

Taking account of the known facts about the $R_{ij}^k$, together with
the invariance of the frame, we have 
\begin{align*}
\clift{\check{X}_r}&=
-R^s_{ri}v^i\vf{v^s}=\left(\Upsilon_{Ir}^sv^I-\bar{C}^s_{rt}v^t\right)\vf{v^s}\\
\vlift{\check{X}_r} &=\vf{v^r}\\
\clift{\check{X}_I} &=
Y_I - R^K_{IJ}v^J\vf{v^K} -R^r_{Ij}v^j\vf{v^r}\\
&=Y_I - R^K_{IJ}v^J\vf{v^K} +
\left(K^r_{IJ}v^J-\Upsilon_{Is}^rv^s\right)\vf{v^r}\\
\vlift{\check{X}_I}  &= \vf{v^I}.
\end{align*}
By substituting these expression in the reduced equations
$\check{\Gamma}(\vlift{\check{X}_i}(l))-\clift{\check{X}_i}(l)=0$
(with $i=r$ and $i=I$ successively) we get:

\begin{prop}
The Lagrange-Poincar\'e equations are given by
\begin{align*}
&\check{\Gamma}\left(\fpd{l}{v^r}\right) =
\left(\Upsilon_{Ir}^sv^I-\bar{C}^s_{rt}v^t\right)\fpd{l}{v^s}\\
&\check{\Gamma}\left(\fpd{l}{v^I}\right) -Y_I(l)
+R^K_{IJ}v^J\fpd{l}{v^K}=
\left(K^r_{IJ}v^J-\Upsilon_{Is}^rv^s\right)\fpd{l}{v^r}.
\end{align*}
\end{prop}
The first of these is of course just the reduced form of the momentum
equation given earlier.

If we take the $Y_I$ to be coordinate fields the horizontal equation
takes on a somewhat more familiar appearance:
\[
\check{\Gamma}\left(\fpd{l}{v^I}\right) -\fpd{l}{x^I}=
\left(K^r_{IJ}v^J-\Upsilon_{Is}^rv^s\right)\fpd{l}{v^r};
\]
now the $v^I$ are effectively the standard fibre coordinates on
$T(Q/G)$ (quasi no longer).

Finally, we reconcile this version of the horizontal
Lagrange-Poincar\'{e} equation (for any frame $\{Y_I\}$ on $Q/G$) with
the more abstract one given earlier,
\[
\check\Gamma(\langle \vlift{Y},\hchecklift{d}l\rangle)-
\langle \clift{Y},\hchecklift{d}l\rangle=0,
\]
by computing $\hchecklift{d}l$ in terms of the frame $\{Y_I\}$.  Let
$\{\vartheta^I\}$ be the basis of 1-forms on $Q/G$ dual to the $Y_I$,
and $v^I$ the quasi-velocities.  Then $\{\vartheta^I,dv^I\}$ is a
basis of 1-forms on $T(Q/G)$ (we haven't distinguished notationally
between 1-forms on $Q/G$ and their pullbacks to $T(Q/G)$).  Note that
\[
\langle \clift{Y_I},\vartheta^J\rangle=\delta^J_I,\quad
\langle \clift{Y_I},dv^J\rangle=-R^j_{IK}v^K,\quad
\langle \vlift{Y_I},\vartheta^J\rangle=0,\quad
\langle \vlift{Y_I},dv^J\rangle=\delta^J_I.
\]
Now $X_I=\hlift{Y_I}$, from which it follows that
\begin{align*}
\hchecklift{(\clift{Y_I})}&=\clift{\check{X}_I}=
Y_I - R^K_{IJ}v^J\vf{v^K} +
\left(K^r_{IJ}v^J-\Upsilon_{Is}^rv^s\right)\vf{v^r}\\
\hchecklift{(\vlift{Y_I})}&=\vlift{\check{X}_I}=\vf{v^I}.
\end{align*}
Recall that $\hchecklift{d}l$ is a 1-form along the projection
$(TQ)/G\to T(Q/G)$, which means that it may be expressed as a linear
combination of the forms $\{\vartheta^I,dv^I\}$ with coefficients
which are functions on $(TQ)/G$.  Using the expressions above for
$\hchecklift{(\clift{Y_I})}$ and $\hchecklift{(\vlift{Y_I})}$ we
obtain
\[
\hchecklift{d}l=\left(Y_I(l)+
\left(K^r_{IJ}v^J-\Upsilon_{Is}^rv^s\right)\fpd{l}{v^r}\right)\vartheta^I
+\fpd{l}{v^I}dv^I,
\]
which leads to the expressions for the reduced equations given
above. 

\subsection{The Lagrange-d'Alembert-Poincar\'{e} equations}

We now choose the anholonomic frame $\{X_i\}$ in the form
$\{X_\alpha,X_a\}$ where $\{X_\alpha\}$ is a local basis of the
distribution $\D$.  We write the quasi-velocities as $(v^\alpha,v^a)$.
The constraint submanifold $\C$ is then simply given by $v^a=0$.  The
Lagrange-d'Alembert equations are
$\Gamma(\vlift{X_\alpha}(L))-\clift{X_\alpha}(L)=0$ on $\C$.
Recall that $\Gamma$ represents here a vector field on $\C$ of
second-order type, which means that it is of the form
\[
\Gamma = v^\alpha \clift{X_\alpha} + f^\alpha\vlift{X_\alpha}.
\]
The Lagrange-d'Alembert equations become
\[
\Gamma\left(\fpd{L}{v^\alpha}\right)
-X_\alpha(L)+R^i_{\alpha\beta}v^\beta\fpd{L}{v^i}=0
\]
in Hamel form.  It is sometimes considered preferable to separate out
those terms which involve differentiation along $\C$ from those which
involve differentiation transverse to it; in the former we can replace
$L$ by $L_c$, the constrained Lagrangian, in other words the
restriction of $L$ to $\C$. We obtain
\[
\Gamma\left(\fpd{L_c}{v^\alpha}\right)
-X_\alpha(L_c)+R^\beta_{\alpha\gamma}v^\gamma\fpd{L_c}{v^\beta}=
-\left.R^a_{\alpha\beta}v^\beta\fpd{L}{v^a}\right|_{\C}.
\]

To obtain reduced equations for an invariant constrained system we
need an adapted frame $\{X_i\}=\{X_{\alpha},X_{a}\}$ which is
invariant, as before.  The basis $\{X_\alpha\}$ of $\D$, in turn, is
of the form $\{X_\rho,X_\kappa\}$ where $\{X_\rho\}$ is a basis for
$\S$.  The set $\{X_a\}$ takes the form $\{X_c,X_k\}$ where the $X_c$
are vertical.  The collection $\{X_\rho,X_c\}$ is a basis $\{X_r\}$ of
the vertical vector fields; in general we can no longer take
$\hat{E}_r$ for $X_r$.  The collection $\{X_\kappa,X_k\}=\{X_I\}$ is
transverse to the fibres of $Q\to Q/G$ and is invariant, so can be
taken to be the horizontal lifts of their projections $Y_I$ to $Q/G$
with respect to some suitable principal connection $\omega$.  The
vector fields $Y_\kappa$ form a basis for $\bar{\D}$.

The corresponding quasi-velocities are $(v^\alpha,v^a)$ or
$(v^\rho,v^\kappa,v^c,v^k)$; the constraint submanifold $\C$ is given by
$v^a=0$, and $(v^\kappa,v^k)$ are quasi-velocities on $Q/G$
corresponding to the frame $\{Y_I\}=\{Y_\kappa,Y_k\}$, with $v^k=0$
defining the constraint submanifold $\bar{\C}$.

The Lagrange-d'Alembert equations reduce, taking $\alpha=\rho$ and
$\alpha=\kappa$ in turn, to the following pair of equations on $\C/G$:
\begin{align*}
&\check{\Gamma}\left(\fpd{l}{v^\rho}\right) =
\left(\Upsilon_{\kappa\rho}^rv^\kappa-\bar{C}^r_{\rho\sigma}v^\sigma\right)
\fpd{l}{v^r}\\
&\check{\Gamma}\left(\fpd{l}{v^\kappa}\right) -Y_\kappa(l)
+R^I_{\kappa\lambda}v^\lambda\fpd{l}{v^I}=
\left(K^r_{\kappa\lambda}v^\lambda-\Upsilon_{\kappa\rho}^rv^\rho\right)
\fpd{l}{v^r}.
\end{align*}
Now $L$ and $\C$ are both invariant under $G$, and so the constrained
Lagrangian $L_c$ is invariant under $G$, and defines a function $l_c$
on $\C/G$, which coincides with the restriction of $l$ (a function on
$(TQ)/G$).  The function $l_c$ is called the constrained reduced
Lagrangian (but might just as well be called the reduced constrained
Lagrangian).  We can use this to rewrite the reduced equations.
\begin{prop}
The
Lagrange-d'Alembert-Poincar\'{e} equations are given by
\begin{align*}
\check{\Gamma}\left(\fpd{l_c}{v^\rho}\right) &=
\left(\Upsilon_{\kappa\rho}^rv^\kappa-\bar{C}^r_{\rho\sigma}v^\sigma\right)
\left.\fpd{l}{v^r}\right|_{\C/G}\\
\check{\Gamma}\left(\fpd{l_c}{v^\kappa}\right) &-Y_\kappa(l_c)
+R^\lambda_{\kappa\mu}v^\mu\fpd{l_c}{v^\lambda}\\
&=-R^k_{\kappa\lambda}v^\lambda\left.\fpd{l}{v^k}\right|_{\C/G}
+\left(K^r_{\kappa\lambda}v^\lambda-\Upsilon_{\kappa\rho}^rv^\rho\right)
\left.\fpd{l}{v^r}\right|_{\C/G}.
\end{align*}
\end{prop}
The first of these is the reduced momentum equation.  The restriction
of the momentum $p$ to $\g^\D$, that is, the map $p^\D:TQ\to(\g^\D)^*$
given by
\[
\langle \xi,p^{\D}(q,u)\rangle=\vlift{\tilde{\xi}_q}(L)(q,u)
\quad\mbox{for $\xi\in\g^q$},
\]
is sometimes called the nonholonomic momentum map.  Its components are
the $G$-invariant functions
$P_\rho=\langle\xi_\rho,p\rangle=\vlift{X_\rho}(L)$, where
$\xi_\rho=X_\rho^rE_r$ defines a section of $\bar{\g}^\D\to Q/G$.  The
functions $P_\rho$ satisfy
$\Gamma(P_\rho)=\bar{X}^r_s\dot{X}_\rho^sP_r$ on $\C$, and this
reduces to the Lagrange-d'Alembert-Poincar\'{e} equation for momentum
given above.

We next consider some special cases of Lagrange-Poincar\'{e}-type
reduction of nonholonomic systems.

In case $T_qQ = \D_q + \V_q$ (i.e.\ when the so-called `dimension
assumption' is satisfied), the space $\bar{\C}$ is the whole of
$T(Q/G)$.  Furthermore, we can replace the $Y_I$ with coordinate
fields $\partial/\partial x^I$ on $Q/G$, and the horizontal reduced
equation becomes
\[
\check{\Gamma}\left(\fpd{l_c}{v^I}\right) -\fpd{l_c}{x^I}=
\left(K^r_{IJ}v^J-\Upsilon_{I\rho}^rv^\rho\right)
\left.\fpd{l}{v^r}\right|_{\C/G}.
\]

In case $\S_q=\{0\}$ the constraints are said to be purely kinematic
in \cite{Cortes}.  In this case there is no momentum equation.

Chaplygin systems (see e.g.\ \cite{Frans,Koiller}) are systems which
have both of the above properties.  There is no momentum equation, and
$\D$ is now the horizontal distribution $\H$ of a principal
connection.  We can therefore identify $\C/G$ with $T(Q/G)$.  The
reduced vector field is now of the form
\[
\check\Gamma = v^I\vf{x^I} + \Gamma^I\vf{v^I},
\]
i.e.\ it is a (true) second-order differential equation field on $Q/G$, and its
coefficients $\Gamma^I$ can be determined from the equations
\[
\check\Gamma\left(\fpd{l_c}{v^I}\right)-\fpd{l_c}{x^I} =
\left.K^r_{IJ}v^J\fpd{l}{v^r}\right|_{T(Q/G)}.
\]
These equations are of the form of Euler-Lagrange equations subjected
to an external force of gyroscopic type.  See e.g.\ \cite{nonholvak}
for more details on this case, in the framework of anholonomic frames.

In case $\D\subset\V$ there is no horizontal equation, and the
momentum equation is just
\[
\check{\Gamma}\left(\fpd{l_c}{v^\rho}\right) =
-\bar{C}^r_{\rho\sigma}v^\sigma\left.\fpd{l}{v^r}\right|_{\C/G}.
\]
One important special case occurs when the configuration space $Q$ is
a Lie group (that is, $Q=G$), and the constraints are linear; the
reduced equations are then called the Euler-Poincar\'{e}-Suslov
equations in e.g.\ \cite{AMZ,suslov}.

\section{Routh-type reduction for systems with horizontal symmetries}

We now consider the class of systems with a so-called group of
horizontal symmetries \cite{BKMM}.  For that case, one assumes that
there exists a subgroup $H \subset G$, the so-called group of
horizontal symmetries, such that $\tilde A \in\D$ for all $A\in\h$
and $\S_q=\D_q \cap \V_q = \V^H_q = \{ \tilde A(q) \,|\, A\in\h,
q\in Q \}$.  Because of the property $\g^{\psi_q(q)}=\ad(g^{-1})\g^q$
that we encountered in Section~\ref{inv}, we get that
$\h=\ad(g^{-1})\h$, meaning that $\h$ is necessarily an ideal (or that
$H$ is a normal subgroup).

For systems with the above properties, one can, of course, still use
the reduction procedure as described in the previous sections.  There
are, however, also other approaches to reduction.  For example, in
\cite{Cortes} it is shown that a version of Marsden-Weinstein
reduction can be applied to this case.  The goal of this section is to
show that one can easily stay on the `Lagrangian side' and recast
everything in terms of Routh reduction.  We will follow the geometric
approach to non-Abelian Routh reduction we have developed in
\cite{Routh}. (For a different approach see \cite{MRS}.)  As before,
the main observation is that all one needs to do is to choose an
appropriate frame.  Let's assume for simplicity that in this section
$\V_q +\D_q = T_qQ$.

Let $\{X_\kappa\}$ be the invariant vector fields we had before.  If
$\{E_r\}=\{E_\rho,E_c\}$ is a basis of $\g$ whose first members
$\{E_\rho\}$ span $\h$, we can use $\{X_\alpha\}=\{X_\kappa,{\tilde
E}_\rho\}$ as a (now not-invariant) anholonomic frame for $\D$, and
$\{X_\kappa,{\tilde E}_\rho,{\tilde E}_c\}$ as a complete basis of
vector fields on $Q$ (with corresponding quasi-velocities
$(v^\kappa,{\tilde v}^\rho,{\tilde v}^c)$.  Given that $\clift{{\tilde
E}_\rho}(L)=0$, the Lagrange-d'Alembert equation in the direction of
${\tilde E}_\rho$ now becomes the conservation law
\[
\vlift{\tilde E}_\rho(L) = \mu_\rho,
\]
where $\mu=\mu_\rho E^\rho\in\h^*$, where the $E^\rho$ are part of the
basis that is dual to $\{ E_\rho,E_c \}$. This represents a relation
on $\C$, not on the whole of $TQ$.

The remaining Lagrange-d'Alembert equations are of the form
\[
\Gamma(\vlift{X_\kappa}(L))-\clift{{X}_\kappa}(L) =0.  \] We will
restrict these equations to a fixed level set of momentum, from now on
denoted by $N_\mu$, and rewrite them in a form that contains only
vector fields that are tangent to $N_\mu$.  We will do so in two steps.
It is easy to see that the vector fields
\[
\clift{\bar X}_\kappa = \clift{X}_\kappa +
{\tilde R}^r_{\kappa\lambda}v^\lambda \vlift{\tilde E}_r
\qquad\mbox{and}\qquad \vlift{X}_\kappa
\]
are tangent to $\C$.  Here ${\tilde R}^r_{\kappa\lambda}$ stands for
the component of the bracket $[X_{\kappa},X_\lambda]$ along ${\tilde
E}_r$.  There is no contribution in ${\tilde v}^\rho$ since
$[X_{\kappa},{\tilde E}_\rho]=0$.  The equations then become
\[\Gamma(\vlift{X_\kappa}(L))-\clift{{\bar X}_\kappa}(L) =- {\tilde
R}^r_{\kappa\lambda}v^\lambda \vlift{\tilde E}_r(L)
\]
on $N_\mu$.

If we further assume the matrix $(g_{\rho\sigma}) = ( \vlift{\tilde
E}_\rho (\vlift{\tilde E}_\sigma(L)) )$ to be non-singular, the
momentum equations can be solved in the form ${\tilde v}^\rho =
\iota^\rho$, where $\iota^\rho$ are functions of the other
variables.  Moreover, under that assumption we can always find vector
fields $\clift{W}_\kappa$ and $\vlift{W}_\kappa$, with
\begin{align*}
\clift{W_\kappa}&=\clift{{\bar X}_\kappa}+A^\rho_\kappa\vlift{\tilde{E}_\rho}\\
\vlift{W_\kappa}&=\vlift{X_\kappa}+B^\rho_\kappa\vlift{\tilde{E}_\rho},
\end{align*}
which, as well as being tangent to $\C$, are also tangent to the level
set $N_\mu$ (the notation may again be a bit misleading since they
will not be complete or vertical lifts).  Let us denote $p_\rho =
\vlift{{\tilde E}_\rho}(L)$ and let us introduce the Routhian of $L$
as the function ${\mathcal R}=L-{\tilde v}^\rho p_\rho$.  Then, on the
level set (which is a part of $\C$)
\begin{align*}
\clift{W_\kappa} ({\mathcal R}) &= \clift{W_\kappa}(L)
- \clift{W_\kappa}({\tilde v}^\rho) p_\rho
= \clift{{\bar X}_\kappa}(L) + A^\rho_\kappa p_\rho
- \clift{{\bar X}_\kappa} ({\tilde v}^\rho) p_\rho - A^\rho_\kappa p_\rho \\
&= \clift{{\bar X}_\kappa}(L) - \clift{{\bar X}_\kappa}({\tilde v}^\rho)p_\rho =
\clift{{\bar X}_\kappa}(L) + {\tilde R}^\rho_{\kappa\beta}v^\beta \mu_\rho
- {\tilde R}^r_{\kappa\lambda}v^\lambda \delta^\rho_r \mu_\rho \\
&= \clift{{\bar X}_\kappa}(L) +
{\tilde R}^\rho_{\kappa\lambda}v^\lambda \mu_\rho
-{\tilde R}^\rho_{\kappa\lambda}v^\lambda \mu_\rho
= \clift{{\bar X}_\kappa}(L) \\
\vlift{W}_\kappa ({\mathcal R}) &= \vlift{W}_\kappa(L)
- \vlift{W}_\kappa({\tilde v}^\rho) p_\rho = \vlift{X}_\kappa(L)
+ B^\rho_\kappa p_\rho
-   B^\rho_\kappa p_\rho \\
&= \vlift{{X}_\kappa}(L).
\end{align*}
Here, ${\tilde R}^\rho_{\kappa\lambda}$ stands for the component of
$[X_\kappa,X_\lambda]$ along ${\tilde E}_\rho$.  We have also used
that $[X_\kappa,{\tilde E}_\rho] = 0$.

The vector fields $\clift{{\tilde E}_\rho}$ are tangent to $\C$.  We
can fix functions $C^\sigma_\rho$ such that the vector fields
\[
\clift{\bar{E}_\rho}=\clift{{\tilde E}_\rho}+
C^\sigma_\rho\vlift{\tilde{E}_\sigma}
\]
are tangent to $N_\mu$.

Since $\Gamma$ is tangent to $p_\rho=\mu_\rho$
its restriction to this level set is of the form
\[
\Gamma = v^\kappa \clift{W_\kappa} + \iota^\rho \clift{\bar{E}_\rho}
+(\Gamma^\kappa\circ\iota) \vlift{W_\kappa}.
\]
The coefficient $\Gamma^\kappa\circ\iota$ can be determined from
the remaining Lagrange-d'Alembert equations, which take the form \[
\Gamma(\vlift{W}_I({\mathcal R}^\mu))
-\clift{W_I}({\mathcal R}^\mu)
=-{\tilde R}^c_{\kappa\lambda}v^\lambda\vlift{{\tilde E}_c}(L)
- {\tilde R}^\rho_{\kappa\lambda}v^\lambda\mu_\rho
\]
on $N_\mu$.

We can now try to understand how to reduce this restriction of
$\Gamma$.  It is easy to see that the action of $G$ on $\C$ can be
restricted to an action of the isotropy group $H_\mu$ on the level set
$N_\mu$ in $\C$.  Indeed, we have
\[
0=A^\sigma\clift{\tilde E}_\sigma(\vlift{\tilde E}_\rho(L)) =
- A^\sigma C^\tau_{\sigma\rho} \vlift{\tilde E}_\tau(L)
= - A^\sigma C^\tau_{\sigma\rho} \mu_\tau
\]
if and only if $A=A^\sigma E_\sigma\in\h_\mu$.  We can therefore
reduce the above vector field to a vector field ${\check\Gamma}_1$ on
$N_\mu/H_\mu$.  This reduction method is the direct analogue of the
situation for standard Routh reduction (in the absence of
constraints).

But there is more.  Since we know that $H$ is normal in $G$, the level
set $\vlift{\tilde E}_\rho(L) = \mu_\rho$ has also the following
behaviour
\[
0=A^r\clift{\tilde E}_r(\vlift{\tilde E}_\rho(L)) =
- A^r C^s_{r\rho} \vlift{\tilde E}_s(L) =
- A^r C^\sigma_{r\rho} \vlift{\tilde E}_\sigma(L) =
- A^r C^\sigma_{r\rho} \mu_\sigma
\]
if and only if $A=A^r E_r \in\g_\mu$.
Therefore, the $G$-action on $\C$ restricts in fact to a $G_\mu$-action on the level
set $N_\mu$.  We are now in the situation of a $G$-invariant
vector field $\Gamma$ on a manifold $\C$, which we can restrict to a
$G_\mu$-invariant vector field on $N_\mu$ and which we can reduce to a
vector field ${\check\Gamma}_2$ on $N_\mu/G_\mu$.

The link with the vector field ${\check\Gamma}_1$ of the previous
paragraph is the following.  Instead of doing a direct reduction by
$G_\mu$, one can perform a reduction in two stages.  Indeed, it is
easy to define an action of $G_\mu/H_\mu$ on $N_\mu/H_\mu$ (see also
\cite{Cortes}).  The vector field ${\check\Gamma}_1$ will be invariant
under that action and we can therefore perform a second reduction by
means of its symmetry group $G_\mu/H_\mu$.

We will not write down explicit expressions for these reduced vector
fields and their corresponding differential equations.  Instead, we
will make the situation clear by means of a simple example.

{\bf Example.} Consider the system with $L=\onehalf({\dot x}^2+{\dot
y}^2+{\dot z}^2)$ on $\R^3$ with constraint $\dot{z}=x\dot x$ (a
variation on the theme of a nonholonomic particle).  This example is
taken from \cite{Cortes}, but we will rephrase it in our current
framework.  We have $\D={\rm span}\{ \partial/\partial x + x
\partial/\partial z, \partial/\partial y \}$ and, since the system is
invariant under the $\R^2$-action given by $(r,s) \times (x,y,z)
\mapsto (x,y+r,z+s) $, $\V = {\rm span} \{ \partial/\partial y,
\partial/\partial z \}$.  Therefore, $\S=\D \cap\V = {\rm span} \{
\partial/\partial y \}$.  This coincides with the case where
$H=\R\times\{0\}$, with action $(r,0) \times (x,y,z) \mapsto (x,y+r,z)
$.  Remark that $\D+\V=TQ$.  Quasi-velocities with respect to the
given frame are $v_x=\dot x$, $v_y=\dot y$ and $v_z=\dot z - x \dot
x$.  The vector field $X_\kappa=X=\partial/\partial x + x
\partial/\partial z$ is invariant under the $G$-action.

The preserved momentum is here
\[
\left(\fpd{}{y} \vlift{\right)}(L)=\dot y = \mu.
\]
The remaining equation on $\C$ is $\Gamma(\vlift{X}(L)) -\clift
X(L)=0$.  Since $\vlift X (v_z) = 0$ and $\clift X (v_z) = 0$, both
$\vlift X$ and $\clift X$ are tangent to the constraint, so we can
rewrite that equation as $\Gamma(\vlift{X}(L_c)) -\clift X(L_c)=0$,
with $L_c = \onehalf((1+x^2){\dot x}^2+{\dot y}^2)$.  One easily
verifies that this equation is equivalent with
\[
(1+x^2)\ddot x  - x {\dot x}^2=0.
\]
This equation is evidently $\R^2$-invariant and the reduced vector field is
\[
\check\Gamma = \dot x \fpd{}{x} + \frac{x {\dot x}^2}{1+x^2}\fpd{}{\dot x}.
\]

It is instructive to see how one gets the same result when we use
Routh reduction.  Remark that $H_\mu = \R\times {0}$ and $G_\mu=\R^2$.
Since also $\vlift X (\dot y) = 0$ and $\clift X (\dot y)=0$ the
vector fields $\vlift X=\vlift{\bar X}$ and $\clift X=\clift{\bar X}$
are already tangent to the level set $\dot y = \mu$.  We can therefore
simply re-write the remaining equation as
\[
\Gamma(\vlift{X}({\mathcal R}_\mu^c)) -\clift X({\mathcal R}_\mu^c)=0,
\] where ${\mathcal R}_\mu^c$ is the restriction of the Routhian
${\mathcal R}_\mu$ to the constraints and to the level set.  It is
given by
 \[
 {\mathcal R}_\mu^c = \onehalf ((1+x^2) {\dot x}^2 - \mu^2).
 \]
The Routh equation above is again $(1+x^2)\ddot x - x {\dot x}^2=0$.
We can now do a direct reduction by means of the largest group
$G_\mu$.  We see that $N_\mu/G_\mu = T\R$.  Due to the absence of
gyroscopic-type terms, the $G_\mu$-reduced version of the above
equation will be a genuine Euler-Lagrange equation, with the
$G_\mu$-reduced Routhian as its Lagrangian.  The vector fields $\vlift
X$ and $\clift X$ reduce to the vector fields $\partial/\partial \dot
x$ and $\partial/\partial x$ on $\R$ and the reduction actually
amounts to cancelling the cyclic variables $y$ and $z$ from the above
equation.  A similar reasoning holds for the reduction in two steps.
We have $H_\mu=\R$ and $G_\mu/H_\mu =\R$; the first reduction cancels
$y$ and the second cancels $z$.

\subsubsection*{Acknowledgements}
The first author is a Guest Professor at Ghent University:\ he is
grateful to the Department of Mathematics for its hospitality.  The
second author is a Postdoctoral Fellow of the Research Foundation --
Flanders (FWO).

\end{document}